\documentclass{article}
\usepackage{amssymb}
\usepackage{amsfonts}
\usepackage{amsmath}
\usepackage{amsthm}
\usepackage{amsfonts}
\usepackage[all]{xy}

\usepackage{enumerate}
\usepackage{setspace}
\usepackage{comment}
\usepackage{graphicx}
\usepackage{float}
\usepackage{tikz}
\usepackage{cite}
\tikzstyle{vertex}=[circle, draw, inner sep=0pt, minimum size=6pt]

\usepackage{calc}
\usepackage{pgfplots}
\usepackage{subfigure}
\usepackage[percent]{overpic}

\newtheorem{theorem}{Theorem}
\newtheorem{proposition}[theorem]{Proposition}
\newtheorem{lemma}[theorem]{Lemma}
\newtheorem{corollary}[theorem]{Corollary}

\theoremstyle{definition}

\theoremstyle{remark}

\let\originalleft\left
\let\originalright\right
\renewcommand{\left}{\mathopen{}\mathclose\bgroup\originalleft}
\renewcommand{\right}{\aftergroup\egroup\originalright}

\newcommand{\nrel}[1]{\mathrm{nRel}(#1;p)}
\newcommand{\nrelvalue}[2]{\mathrm{nRel}(#1;#2)}

\title{Maximal Intervals of Decrease and \\Inflection Points for Node Reliability}
\author{Jason Brown\\
Department of Mathematics and Statistics\\
Dalhousie University\\
Halifax, Nova Scotia, Canada B3H 3J5}

\begin{document}
\maketitle
\begin{abstract}
The \textit{node reliability} of a graph $G$ is the probability that at least one node is operational and that the operational nodes can all communicate in the subgraph that they induce, given that the edges are perfectly reliable but each node operates independently with probability $p\in[0,1]$.  We show that unlike many other notions of graph reliability, the number of maximal intervals of decrease in $[0,1]$ is unbounded, and that there can be arbitrarily many inflection points in the interval as well.
\end{abstract}

%\begin{keyword}
%graph, node reliability, maximal interval of decrease, point of inflection
%\end{keyword}

%\doublespacing

Robustness of a network to random failures has been well-studied. In the most usual model, {\em all-terminal reliability}, nodes of an (undirected , finite) graph $G$ are always assumed to operational, while the edges are independently operational with probability $p \in [0,1]$, and we are interested in the probability that the operational edges form a  spanning connected subgraph. This is model is well-behaved, in the sense that  the all-terminal reliability is always an increasing function of $p$ -- making the edges more ``reliable" can on. Many other reliability models both of undirected and directed graphs, such as {\em two-terminal reliability}, {\em $K$-terminal reliability} and {\em reachability} (see, for example, \cite{ColbournBook}) have the same expected property. For a recent discussion of reliability in variants, see \cite{highwayreliability}.

On the other hand, the node version of reliability, {\em node reliability}, behaves less well. Given a graph $G$, we assume now the nodes each are independently operational with probability $p \in [0,1]$, and ask for the probability $\nrel{G}$, that the induced subgraph on the operational vertices is connected (and nonempty). Node reliability has been used to model the robustness of missile deployment networks \cite{Stivaros} as well as that of connectivity in social networks (see, for example, \cite{brownmolshape}). Much of the research into node reliability have been on optimal graphs \cite{NodeRel1,NodeRel2,Stivaros,NodeRel3} and algorithmic issues \cite{NodeRelComplexity1,NodeRelComplexity2}.

Unlike the other forms of reliability, node reliability is not necessarily an increasing function of $p \in [0,1]$. In fact, there can be a significant interval of decrease. Figure~\ref{fig:fivepathDecreasing} shows a plot of the node reliability of $P_5$, the path of order $5$ (i.e. on $5$ nodes), and one can see an interval of decrease of approximately $[0.28,0.59]$.

\begin{figure}[!h]
\centering
\includegraphics[scale=0.6]{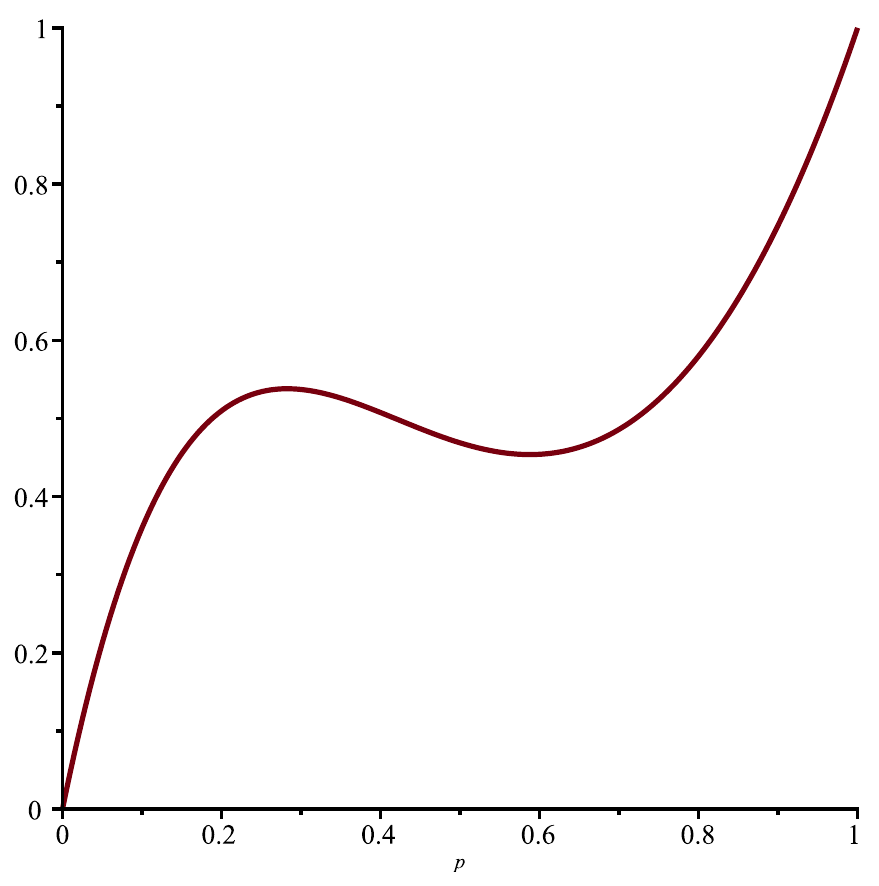}
\caption{A plot of $\nrel{P_{5}} = 3p^5-12p^4+21p^3-16p^2+5p$.}
\label{fig:fivepathDecreasing}
\end{figure}

In previous known cases (and there are many -- the node reliability of any connected graph of order $n$ with at most $0.0851n^2$ edges has an interval of decrease \cite{brownmolshape}), when an interval of decrease was detected, there was just one such maximal interval of decrease.   This begs the question of how pathological node reliability can be with respect to monotonicity -- can there be multiple maximal intervals of decrease in $[0,1]$? This questions, along with the question of whether the number of inflection points in $[0,1]$ for node reliability, were raised in \cite{highwayreliability,brownmolshape} (it was previously known that there is a graph whose node reliability has three points of inflection in the interval). Here we answer both in the affirmative. 

\vspace{0.2in}
We begin with a simple graph operation that will be crucial to our argument. For a graph $G$ and positive integer $l$, we let $G[K_l]$ denote the graph formed from $G$ by replacing each vertex of $G$ by a complete graph of order $l$ (this graph is often called the graph formed from $G$ by substituting in a copy $K_l^v$ of $K_l$ for each node $v$, or the {\em lexicographic product} of $G$ with $K_l$). We shall make repeated use of the following lemma.

\begin{lemma}
For any graph $G$, 
\begin{eqnarray}
\nrel{G[K_l]} & = & \nrelvalue{G}{1-(1-p)^l}. \label{clique}
\end{eqnarray}
\end{lemma}
\begin{proof}
Consider any subset $S$ of $V(G[K_l])$, and $S^\prime$ be the subset of nodes $v$ of $G$ such that  $K_l^v \cap S \neq \emptyset$. Then it is easy to see that the subgraph of $G[K_l]$ induced by $S$ is connected and nonempty if and only if the subgraph of $G$ induced by $S^\prime$ is connected and nonempty. The result easily follows.
\end{proof}

Note that from (\ref{clique}), the effect is that, as $l$ grows, the curve of $\nrel{G[K_l]}$ moves more and more to the left. 

\begin{corollary}\label{moveextremaleft}
Suppose that $G$ is a graph whose node reliability has local extrema 
\[ \beta_1 < \beta_2 < \cdots < \beta_k\]
at locations $\alpha_1 < \alpha_2 < \cdots < \alpha_k$ inside the interval $(0,1)$. Then for any $\gamma \in (0,1)$, there is positive integer $L$ such that for all $l \geq L$, $\nrel{G[K_l]}$ has the same local extrema as $G$ inside $(0,1)$, at locations that are less than $\gamma$. \qed
\end{corollary}

Given graph $F$, we denote the graphs formed from $F$ by either adjoining an isolated vertex or a vertex universal to $F$ (i.e joined to all vertices of $F$) by $F \cup K_1$ and $F + K_1$, respectively. The idea is that we will alternately adjoin isolated and universal vertices to introduce more extrema. The issue is how to do this without destroying old extrema. The next two results uses Corollary~\ref{moveextremaleft} to do so. 

\begin{proposition}\label{K1disjointunion}
Suppose that $G$ is a connected graph, $k$ a positive integer and $\nrelvalue{G}{\alpha_i} = \beta_i$, where 
\[ 0 < \alpha_1 < \alpha_2 < \cdots < \alpha_k < 1\] 
and 
\[ 0 < \beta_1 > \beta_2 < \beta_3 > \beta_4 < \cdots < \beta_k.\] 
Then there is positive integer $l$ such that for some $\alpha_{1}^\prime,\alpha_{2}^\prime,\ldots,\alpha_{k}^\prime,\alpha_{k+1}^\prime \in (0,1)$, that  
\[ 0 < \alpha_1^\prime < \alpha_2^\prime < \cdots < \alpha_k^\prime < \alpha_{k+1}^\prime < 1\] and $\nrelvalue{G[K_l] \cup K_1}{\alpha_{k}^\prime} = \beta_{i}^{\prime}$ satisfy 
\[ 0 < \beta_1^\prime > \beta_2^\prime < \beta_3^\prime < > \beta_4^\prime < \cdots < \beta_k^\prime > \beta_{k+1}^\prime.\]
\end{proposition}

\begin{proof}
To start, we set 
\[ \gamma = \mbox{min}\left( \frac{1}{4},\left\{ \frac{1}{3}\left( \beta_{i+1}-\beta_{i}\right): 1 \leq i \leq k-1\right\}\right).\]
Clearly $\gamma$ is in $(0,1)$. From Corollary~\ref{moveextremaleft}, if $g_{l}(p) = 1-(1-p)^{1/l}$ (the inverse of $f_{l}(p) = 1-(1-p)^l$ on $[0,1]$) we can chose an $l$ such that $\alpha_{i}^\prime = g_l(\alpha_{i}) < \gamma$ for $i = 1,\ldots, k$. Then from (\ref{clique}), we have that
\[ \nrelvalue{G[K_l]}{\alpha_{i}^\prime} = \nrelvalue{G}{\alpha_{i}} = \beta_{i}\]
for $i = 1,\ldots,k$.

Consider now the graph $H = G[K_l] \cup K_1$; let its order be $n$. Its node reliability is given by
\[ \nrel{H} = p(1-p)^{n-1} + (1-p)\nrel{G[K_l]}.\]
It follows that for any $i \in \{1,2,\ldots,k\}$,
\begin{eqnarray*}
\nrelvalue{H}{\alpha_{i}^\prime} & = & \alpha_{i}^\prime (1-\alpha_{i}^\prime)^{n-1} + (1-\alpha_{i}^\prime)\nrelvalue{G[K_l]}{\alpha_{i}^\prime}\\
                                                   & = & \alpha_{i}^\prime (1-\alpha_{i}^\prime)^{n-1} + (1-\alpha_{i}^\prime)\beta_{i}.
\end{eqnarray*}
We therefore define 
\[ \beta_{i}^{\prime} = \nrelvalue{H}{\alpha_{i}^\prime} = \alpha_{i}^\prime (1-\alpha_{i}^\prime)^{n-1} + (1-\alpha_{i}^\prime)\beta_{i}.\]

Now for $i = 1,\ldots,k-1$,
\begin{eqnarray*}
\beta_{i+1}^\prime - \beta_{i}^\prime & = & (\alpha_{i+1}^\prime (1-\alpha_{i+1}^\prime)^{n-1} + (1-\alpha_{i+1}^\prime)\beta_{i+1}) - (\alpha_{i}^\prime (1-\alpha_{i}^\prime)^{n-1} + (1-\alpha_{i}^\prime)\beta_{i})\\
 & = & (\alpha_{i+1}^\prime (1-\alpha_{i+1}^\prime)^{n-1} - \alpha_{i}^\prime (1-\alpha_{i}^\prime)^{n-1}) + (1-\alpha_{i+1}^\prime)(\beta_{i+1}-\beta_{i}) +(\alpha_{i}^\prime - \alpha_{i+1})\beta_{i}
\end{eqnarray*}
Now $1-\alpha_{i+1}^\prime \geq 3/4$, and
\[ |\alpha_{i+1}^\prime (1-\alpha_{i+1}^\prime)^{n-1} - \alpha_{i}^\prime (1-\alpha_{i}^\prime)^{n-1}| \leq \alpha_{i+1}^\prime \leq \gamma \leq (1/4)(\beta_{i+1}-\beta_{i}),\]
and
\[ |(\alpha_{i}^\prime - \alpha_{i+1})\beta_{i}| \leq \alpha_{i+1} \leq \gamma \leq (1/4)(\beta_{i+1}-\beta_{i}),\]
it follows that $\beta_{i+1}^\prime - \beta_{i}^\prime$ will have the same sign as $\beta_{i+1}-\beta_{i}$, so that
\[ 0 < \beta_1^\prime > \beta_2^\prime < \beta_3^\prime < > \beta_4^\prime < \cdots < \beta_k^\prime.\]
Finally, as $H = G[K_l] \cup K_1$ is disconnected, $\nrelvalue{H}{1} = 0$, and so we can choose $\alpha_{k+1}^\prime \in (\gamma,1)$ such that 
\[ \beta_{k+1}^\prime = \nrelvalue{H}{\alpha_{k+1}^\prime} < \beta_{k}^\prime,\]
and we are done.
\end{proof}

A similar argument shows the following.

\begin{proposition}\label{K1join}
Suppose that $G$ is a disconnected graph, $k$ a positive integer and $\nrelvalue{G}{\alpha_i} = \beta_i$, where 
\[ 0 < \alpha_1 < \alpha_2 < \cdots < \alpha_k < 1\] 
and 
\[ 0 < \beta_1 > \beta_2 < \beta_3 > \beta_4 < \cdots > \beta_{k}.\] 
Then there is positive integer $l$ such that for some $\alpha_{1}^\prime,\alpha_{2}^\prime,\ldots,\alpha_{k}^\prime,\alpha_{k+1}^\prime \in (0,1)$, that  
\[ 0 < \alpha_1^\prime < \alpha_2^\prime < \cdots < \alpha_k^\prime < \alpha_{k+1}^\prime < 1\] and $\nrelvalue{G[K_l] + K_1}{\alpha_{k}^\prime} = \beta_{i}^{\prime}$ satisfy 
\[ 0 < \beta_1^\prime > \beta_2^\prime < \beta_3^\prime < > \beta_4^\prime < \cdots >\beta_k^\prime < \beta_{k+1}^\prime.\]
\end{proposition}
\begin{proof}
We define $\gamma$ and $\alpha_{i}^\prime$ ($i = 1,\ldots,k$) as in the previous proof, and have that
\[ \nrelvalue{G[K_l]}{\alpha_{i}^\prime} = \nrelvalue{G}{\alpha_{i}} = \beta_{i}\]
for $i = 1,\ldots,k$.

Consider now the graph $H = G[K_l] + K_1$. Its node reliability is given by
\[ \nrel{H} = p + (1-p)\nrel{G[K_l]}.\]
It follows that for any $i \in \{1,2,\ldots,k\}$,
\begin{eqnarray*}
\nrelvalue{H}{\alpha_{i}^\prime} & = & \alpha_{i}^\prime + (1-\alpha_{i}^\prime)\nrelvalue{G[K_l]}{\alpha_{i}^\prime}\\
                                                   & = & \alpha_{i}^\prime + (1-\alpha_{i}^\prime)\beta_{i}.
\end{eqnarray*}
We define again
\[ \beta_{i}^{\prime} = \nrelvalue{H}{\alpha_{i}^\prime} = \alpha_{i}^\prime + (1-\alpha_{i}^\prime)\beta_{i}.\]

Now for $i = 1,\ldots,k-1$,
\begin{eqnarray*}
\beta_{i+1}^\prime - \beta_{i}^\prime & = & (\alpha_{i+1}^\prime + (1-\alpha_{i+1}^\prime)\beta_{i+1}) - (\alpha_{i}^\prime + (1-\alpha_{i}^\prime)\beta_{i})\\
 & = & (\alpha_{i+1}^\prime - \alpha_{i}^\prime) + (1-\alpha_{i+1}^\prime)(\beta_{i+1}-\beta_{i}) +(\alpha_{i}^\prime - \alpha_{i+1}^\prime)\beta_{i}
\end{eqnarray*}
A very similar argument to that in the previous proof that $\beta_{i+1}^\prime - \beta_{i}^\prime$ will have the same sign as $\beta_{i+1}-\beta_{i}$, so that
\[ 0 < \beta_1^\prime > \beta_2^\prime < \beta_3^\prime > \beta_4^\prime < \cdots > \beta_k^\prime.\]
Finally, as $H = G[K_l] \cup K_1$ is connected, $\nrelvalue{H}{1} = 1$, and so we can choose $\alpha_{k+1}^\prime \in (\gamma,1)$ such that 
\[ \beta_{k+1}^\prime = \nrelvalue{H}{\alpha_{k+1}^\prime} > \beta_{k}^\prime,\]
and we are done.
\end{proof}

\begin{theorem}
Let $k$ be any positive integer. There is a connected graph $G$ whose node reliability has at least $k$ maximal intervals of decrease in $(0,1)$.
\end{theorem}
\begin{proof}
If $k=1$, we can take $G = P_{5}$ (in fact , any one of many graphs with an interval of decrease in $(0,1)$ can be a starting point for our induction). For any $k > 1$, let $H$ be any connected graph whose node reliability has (at least) $k-1$ maximal intervals of decrease. Then clearly there are $\alpha_{1},\alpha_{2},\ldots,\alpha_{2k-2}$ with 
\[ 0 < \alpha_{1} < \alpha_{2} < \ldots < \alpha_{2k-2} < 1\] 
with 
\[ 0 < \beta_{1} > \beta_2 < \beta_3  > \beta_4 < \ldots < \beta_{2k-3} > \beta_{2k-2} < 1\] 
and $\beta_{i} = \nrelvalue{H}{\alpha_{i}}.$
From Proposition~\ref{K1disjointunion}, we can blow up the graph by substituting a large enough clique in for each vertex of $H$ and then adjoin an isolated vertex to get a disconnected graph whose node reliability has another local maximum tacked on the end. Call this new graph $H^\prime$. Then by Proposition~\ref{K1join}, we can substitute a large enough clique in for each vertex of $H^\prime$ and then adjoin a universal vertex to get a new graph $G$ whose node reliability will have another local minimum at the end. It follows that the node reliability of the new graph $G$ has at least one more interval of decrease than $H$ had, and we are done.
\end{proof}

Another interesting question that has been posed (see \cite{brownmolshape}) is whether there can be arbitrarily many points of inflection for node reliability (for all-terminal reliability, the answer is known, the affirmative, but only when multiple edges are allowed \cite{ManyInflectionPoints}). It is easy to see that if a polynomial $f = f(p)$ has a local maximum at $p = a$, then  $f^\prime(x)$ is increasing sufficiently close but to the left of $a$ and decreasing sufficiently close but to the right of $a$, 
$f$ is concave down for all $x$ sufficiently close to (but not necessarily equal to) $a$. Likewise, around a local minimum, $f$ is concave up. 
It follows that if the node reliability of a connected graph $G$ has $k$ maximal intervals of decrease, then $G$ has at least $k$ points of inflection. We thus derive the following interesting consequence.

\begin{corollary}
For any positive integer $k$, there are (connected) graphs whose node reliabilities has at least $k$ points of inflection. \qed
\end{corollary}

\section*{Acknowledgements}
 
The author would like to thank J. Janssen for her insightful comments. This research is partially supported by grant RGPIN-2018-05227 from Natural Sciences and Engineering Research Council of Canada (NSERC).  

\singlespacing

\bibliographystyle{amsplain}

\providecommand{\bysame}{\leavevmode\hbox to3em{\hrulefill}\thinspace}
\providecommand{\MR}{\relax\ifhmode\unskip\space\fi MR }
\providecommand{\MRhref}[2]{%
  \href{http://www.ams.org/mathscinet-getitem?mr=#1}{#2}
}
\providecommand{\href}[2]{#2}

%\bibliographystyle{elsart-num-sort}
%\biboptions{sort&compress}
%\bibliography{NodeRelShape}

\end{document}